\documentclass[a4paper,11pt]{article}
\usepackage[english]{babel} 
\usepackage[centertags,intlimits]{amsmath}
\usepackage{amsfonts}
\usepackage{amsthm}
\usepackage{amssymb}

\newtheorem{theorem}{Theorem}[section]
{Proposition}
\newtheorem{definition}{Definition}[section]

\newtheorem{lemma}[theorem]%
{Lemma}
\newtheorem{corollary}[theorem]%
{Corollary}

\newtheorem{example}{Example}[section]

\begin{document}

\title{On the difference between proximity and other distance parameters in triangle-free graphs 
      and $C_4$-free graphs}

\author{ Peter Dankelmann\footnote{Financial support
by the South African National Research Foundation, grant 118521, is
gratefully acknowledged.}, 
Sonwabile Mafunda\footnote{The results of this paper form part of the 
second author's PhD thesis. Financial support by the nGAP Programme of the South African 
Department of Higher Education is gratefully acknowledged.}.    \\
University of Johannesburg}


\maketitle

\begin{abstract}
The average distance of a vertex $v$ of a connected graph $G$ is the arithmetic mean of the distances 
from $v$ to all other vertices of $G$. The proximity $\pi(G)$ and the remoteness $\rho(G)$ of $G$ 
are the minimum and the maximum of the average distances of the vertices of $G$.
In this paper, we give upper bounds on the difference between the remoteness and proximity, 
the diameter and proximity, and the radius and proximity of a triangle-free graph with given order and 
minimum degree. We derive the latter two results by first proving lower bounds on the proximity 
in terms of order, minimum degree and either diameter or radius. Our bounds are sharp apart from
an additive constant. 
We also obtain corresponding bounds for $C_4$-free graphs.
\end{abstract}

MSC-class: 05C12

\section{Introduction}
Let $G$ be a finite, connected graph of order $n \geq 2$ with vertex set $V(G)$. 
The \textit{average distance} $\overline{\sigma}(v)$ of a vertex $v\in V(G)$ is defined as the 
arithmetic mean of the distances from $v$ to all other vertices of $G$, i.e. 
$\overline{\sigma}(v)=\frac{1}{n-1} \sum_{w \in V(G)} d(v,w)$, where $d(v,w)$ denotes the usual 
shortest path distance. 
The  {\em proximity} $\pi(G)$ is defined as $\min_{v\in V(G)} \overline{\sigma}(v)$, and the {\em remoteness} $\rho(G)$ of $G$ is defined as $\max_{v\in V(G)} \overline{\sigma}(v)$. 

Bounds on the proximity and the remoteness for graphs were first investigated by Zelinka \cite{ze}, 
and independently by Aouchiche and Hansen \cite{ao}, who investigated relations between these two 
graph invariants and with other invariants, such as diameter, radius, average eccentricity, 
average distance, independence number and matching number.
Among others, they obtained the following three results.

\begin{theorem}{\rm(Aouchiche, Hansen \cite{ao})}\label{thm1.1}
Let $G$ be a connected graph of order $n\geq 3$. Then
\[\rho(G)-\pi(G)\leq \left\{
                \begin{array}{ll}
                  \frac{n-1}{4}\;\quad\qquad\qquad\mbox{if $n$ is odd,}\\
                  \frac{n-1}{4}-\frac{1}{4n-4}\qquad\mbox{if $n$ is even.}\\
                \end{array}
              \right.\] 
Equality holds if and only if $G$ is a graph obtained from a path $P_{\lceil\frac{n}{2}\rceil}$ and 
any connected graph $H$ on $\lfloor\frac{n}{2}\rfloor +1$ vertices by identifying an endpoint of 
the path with any vertex of $H$.
\end{theorem}

\begin{theorem}{\rm(Aouchiche, Hansen \cite{ao})}\label{thm1.2}
Let $G$ be a connected graph of order $n\geq 3$. Then 
\[{\rm diam}(G)-\pi(G)\leq \left\{
                \begin{array}{ll}
                  \frac{3n-5}{4}\;\quad\qquad\qquad\mbox{if $n$ is odd,}\\
                  \frac{3n-5}{4}-\frac{1}{4n-4}\qquad\mbox{if $n$ is even,}\\
                \end{array}
              \right.\] 
 with equality if and only if $G$ is a path $P_n$.
\end{theorem}

\begin{theorem}{\rm(Aouchiche, Hansen \cite{ao})}\label{thm1.3}
Let $G$ be a connected graph of order $n\geq 3$. Then 
\[{\rm rad}(G)-\pi(G)\leq \left\{
                \begin{array}{cc}
                  \frac{n-1}{4}-\frac{1}{n-1}& \mbox{if $n$ is odd,}\\
                  \frac{n-1}{4}-\frac{1}{4n-4}&\mbox{if $n$ is even.}\\
                \end{array}
              \right.\] 
This bound is sharp. Equality holds, for example, for the graph composed of a cycle with an additional edge forming a triangle or two additional crossed edges on four successive vertices of the cycle if $n$ is odd, and by the path $P_n$ or the cycle $C_n$ if $n$ is even.
\end{theorem}

The results in \cite{ao}  prompted further research on differences between proximity or 
remoteness and other distance parameters. Hua, Chen and Das \cite{hu}  
determined the minimum value of the difference between remoteness and radius, thus proving a 
conjecture from \cite{ao}. A sharp upper bound on the difference between proximity and 
average eccentricity was given by Ma, Wu and Zhang \cite{ma}. Such bounds for trees had been
obtained by Sedlar \cite{se}).  For further results on 
proximity and remoteness, see for example \cite{ai, aou, ba, cz, cza, da, dan, lin, wu}.

The starting point for this paper is a strengthening of the bounds in 
Theorems \ref{thm1.1}, \ref{thm1.2} and \ref{thm1.3} that takes into account also the 
minimum degree.

\begin{theorem}{\rm(Dankelmann \cite{dan})}\label{thm1.4}
Let $G$ be a connected graph of order $n$ and minimum degree $\delta$, with $\delta\geq 2$. Then
\[\rho(G)-\pi(G)\leq \frac{3n}{4(\delta+1)}+3.\]
This bound is sharp apart from an additive constant.
\end{theorem}

\begin{theorem}{\rm(Dankelmann \cite{dank})}\label{thm1.5}
Let $G$ be a connected graph on $n$ vertices and minimum degree $\delta$ with $n\geq 20$ and $\delta \geq 2$. Then 
\[{\rm diam}(G)-\pi(G)\leq \frac{9}{4(\delta +1)}n +\frac{3}{4}\delta.\] 
This bound is sharp apart from an additive constant.
\end{theorem}

\begin{theorem}{\rm(Dankelmann \cite{dank})}\label{thm1.6}
Let $G$ be a connected graph on $n$ vertices and minimum degree $\delta$ with $\delta<\frac{n}{4}-1$. Then 
\[{\rm rad}(G)-\pi(G)\leq \frac{3}{4(\delta +1)}n +\frac{8\delta +5}{4(\delta +1)}.\] 
This bound is sharp apart from an additive constant.
\end{theorem}

The goal of this paper is to show that the bounds in Theorems \ref{thm1.4}, \ref{thm1.5} and 
\ref{thm1.6} can be strengthened significantly 
for triangle-free graphs, and also for graphs not containing a $4$-cycle as a (not necessarily 
induced) subgraph.

\section{Terminology and Notation}
We use the following notation. 
Let $G$ be a connected graph and let $v$ be a vertex of $G$.
Then the {\em neighbourhood} of $v$, denoted by $N(v)$, is the set of all vertices adjacent to $v$.
The {\em closed neighbourhood} $N[v]$ of $v$ is the set $N(v)\cup \{v\}$.
For a subset $A\subseteq V(G)$ we define $N[A]=\bigcup_{v\in A}N[v]$.
The {\em degree} ${\rm deg}_G(v)$ of a vertex $v\in V(G)$ is the number of vertices in $N_G(v)$, and the {\em minimum degree} of $G$ denoted $\delta (G)$ is the smallest of the degrees of the vertices of $G$.

If $i\in\mathbb{Z}$, then $N_i(v)$ is the set of all vertices at distance $i$ from $v$.
By $N_{\leq i}(v)$ and $N_{\geq i}(v)$ we mean the set of vertices at distance
at most $i$ and at least $i$, respectively, from $v$.
For a subset $A\subseteq V(G)$ the distance between any vertex $v\in V(G)$ and $A$ is defined as ${\rm min}_{a\in A} d_{G}(v,a)$.
For a vertex $w$ and a subset $X$ of $V(G)$, we denote by $\sigma(w|X)$ the total distance of $w$ in $X$, that is, $\sum_{x\in X}d_G(w,x)$. Thus the {\em total distance} (or {\em distance} for short) of a vertex $v$ in $G$, $\sigma(v|V(G))$, is the sum of the distances from $v$ to all other vertices of $G$.
We will drop the $V(G)$ and simply write $\sigma(v)$ if the set at which the distance is evaluated is the entire vertex set of $G$. A {\em median vertex} ({\em margin vertex}) of $G$ is a vertex  
that minimises (maximises) the total distance among the vertices of $G$. 

The {\em eccentricity} of $v$, denoted by ${\rm ecc}(v)$, is the distance from $v$ to a vertex 
farthest from $v$ in $G$. The {\em diameter} ${\rm diam}(G)$ of $G$ is the largest of all eccentricities 
of the vertices of $G$. 
The {\em radius} on the other hand, denoted ${\rm rad}(G)$, is the smallest of all eccentricities of 
the vertices of $G$. A vertex whose eccentricity equals ${\rm rad}(G)$ is called a {\em centre vertex} 
of $G$.

By $K_n$, $\overline{K}_n$, $C_n$ and $P_n$ we mean the complete graph, the edgeless graph, the cycle, 
and the path on $n$ vertices.   By a {\em triangle} we mean the graph $K_3$. If $F$ is a graph, then 
we say that $G$ is $F$-{\em free} if $G$ does not contain $F$ as a (not necessarily induced) subgraph. 

We define the {\em sequential sum} of graphs $G_1+G_2+\cdots + G_k$ of the graphs 
$G_1, G_2,\ldots,G_k$ to be the graph with vertex set $V(G_1)\cup V(G_2)\cup\cdots\cup V(G_k)$
and edge set 
$E(G_1)\cup E(G_2)\cup\cdots\cup E(G_k)
 \cup \{uv|\;u\in V(G_i), v\in V(G_{i+1})\; i \in \{1,2,\ldots, k-1\} \}$. 
If $p\in \mathbb{N}$, then $\left[G_1 + G_2 + G_3 + G_4 \right]^p$ stands for 
$G_1 + G_2 + G_3+ G_4 + G_1 + G_2 + G_3 + G_4 \ldots + G_1 + G_2 + G_3 + G_4$, where the pattern 
$G_1 + G_2 + G_3 + G_4$ is repeated $p$ times.

\section{Remoteness vs proximity}
In this section we show that the bound on the difference between the remoteness and the proximity of 
a graph in Theorem \ref{thm1.4} can be improved significantly for triangle-free graphs and for $C_4$-free 
graphs. The proof strategy is similar to that of Theorem \ref{thm1.4} (see \cite{dan}).

\begin{theorem}\label{thm3.2}
Let $n, \delta\in\mathbb{N}$, with $\delta\geq 3$ and $n\geq 7$. If $G$ is a connected,  
triangle-free graph of order $n$ and minimum degree $\delta$, then 
\[\rho(G)-\pi(G)\leq \frac{n+1}{2\delta}+4.\]
\end{theorem} 

\begin{proof}
{\rm
Let $u$ and $v$ be two vertices of $G$ with $\bar{\sigma}(u)=\pi(G)$ and $\bar{\sigma}(v)=\rho(G)$. 
Let $p=d(u,v)$  and let $P: u_0,u_1,\ldots, u_{p-1}, u_p$ be a shortest $(v,u)$-path in $G$ where 
$u_0=v$ and $u_p=u$. We now define a subset $B$ of $V(G)$ such that $|\sigma(v|B)-\sigma(u|B)|$ is 
small. To define the set $B$, we first define a smaller subset $A_{i}$ of $B$. 

For $i \in \{0,1,\ldots , p\}$ we define $A_i$ to be a subset of $N(u_i)$ with exactly $\delta$ vertices. 
If also $i\neq p$, then we let $B_i = A_i \cup A_{i+1} \cup A_{p-1-i} \cup A_{p-i}$. Clearly 
$A_{i}\cap A_{i+1} = \emptyset = A_{p-i-1} \cap A_{p-i}$ 
since $G$ is triangle-free. Also, $A_i \cap A_j = \emptyset$ if $|i-j| \geq 3$ since otherwise $P$ 
would not be a shortest path from $v$ to $u$ in $G$. 
Hence, the sets $B_{4i}$, $i=0,1, \ldots, \big\lfloor\frac{p-5}{8}\big\rfloor$, are disjoint and
have cardinality $4\delta$. 

Define $B$ to be the set $\bigcup_{i=0}^{\lfloor \frac{p-5}{8} \rfloor} B_{4i}$. 
We bound separately $|\sigma(v|B)-\sigma(u|B)|$ and 
$|\sigma(v|V(G)\setminus B)-\sigma(u|V(G)\setminus B)|$. 
First fix $i$ and consider $|\sigma(v|B_{4i})-\sigma(u|B_{4i})|$. 
Let $x \in B_{4i}$, so $x \in A_j$ for some $j \in \{4i, 4i+1, p-4i-1, p-4i\}$. For this $j$ define 
\[ \epsilon_x = d(v,x) - d(v,u_j)  \quad
\textrm{and} \quad 
\epsilon'_x = d(u,x) - d(u,u_j). \]
Clearly, $-1\leq \epsilon_x\leq 1$
and $-1 \leq \epsilon'_x \leq 1$. Hence  
\begin{eqnarray} 
\big| \sigma(v | B_{4i}) - \sigma(u | B_{4i}) \big|  
 & = & \Big| 
   \big( \sigma(v | A_{4i}) + \sigma(v | A_{4i+1}) + \sigma(v | A_{p-1-4i}) + \sigma(v | A_{p-4i}) \big) \nonumber \\
 & &  -    \big( \sigma(u | A_{4i}) + \sigma(u | A_{4i+1}) + \sigma(u | A_{p-1-4i}) + \sigma(u | A_{p-4i}) \big) 
  \Big|  \nonumber \\
 & \leq & 
    \big| \sigma(v | A_{4i}) - \sigma(u | A_{p-4i}) \big|  
   + \big| \sigma(v | A_{4i+1}) - \sigma(u | A_{p-1-4i}) \big| \nonumber  \\
&   &   +\big| \sigma(v | A_{p-4i-1}) - \sigma(u | A_{4i+1}) \big| 
   + \big| \sigma(v | A_{p-4i}) - \sigma(u | A_{4i}) \big|. \nonumber \\
   & &  \label{eq:rho-pi-1}
\end{eqnarray}  
For $j \in \{4i, 4i+1, p-1-4i, p-4i\}$. 
consider $| \sigma(v | A_j) - \sigma(u | A_{p-j}) | $. 
For $x \in A_j$ we have $d(v,x)= d(v,u_j) + \epsilon_x= j + \epsilon_x$ and 
for $x \in A_{p-j}$ we have $d(u,x)= d(u,u_{p-j}) + \epsilon_x' = j + \epsilon_x'$. Hence 
\begin{eqnarray*}
 \big| \sigma(v | A_j) - \sigma(u | A_{p-j}) \big| 
  &= & \big|  \sum_{x \in A_j}   ( j+ \epsilon_x ) 
      -  \sum_{x \in A_{p-j}}  
     ( j+ \epsilon_x') \big| \\
     & = & \big| \sum_{x \in A_j} \epsilon_x + \sum_{x \in A_{p-j}} \epsilon_x'  \big| \\
     & \leq & \sum_{x \in A_j} |\epsilon_x| + \sum_{x \in A_{p-j}} |\epsilon_x'| \\     
     & \leq & |A_j| + |A_{p-j}|.
\end{eqnarray*}   
Adding this inequality for all $j \in \{4i, 4i+1, p-1-4i, p-4i\}$ , we obtain in 
conjunction with \eqref{eq:rho-pi-1} that 
\[ 
\big| \sigma(v|B_{4i}) - \sigma(u,B_{4i}) \big|  \leq  2\big( |A_{4i}| + |A_{4i+1}| + |A_{p-1-4i}| + |A_{p-4i}| \big) 
  =  2 |B_{4i}|,
\]
and thus, by adding the above inequality for $i=0,1,2,\ldots,\lfloor \frac{p-5}{8} \rfloor$, 
\begin{equation}\label{correct1}
\big|\sigma(v|B)-\sigma(u|B)\big|\leq 2|B|.
\end{equation}
Now consider $V(G)\setminus B$ and let $w\in V(G)\setminus B$. By the triangle inequality
\[  \big| d(v,w) - d(u,w) \big|  \leq p. \]
Hence, 
\begin{align}\label{correct2}
\big|\sigma(v|V(G)\setminus B)-\sigma(u|V(G)\setminus B)\big|&=\Big|\sum_{w\in V(G)\setminus B}d(v,w)-\sum_{w\in V(G)\setminus B}d(u,w)\Big| \nonumber\\
&\leq\sum_{w\in V(G)\setminus B}p \nonumber\\ 
&=p\left(n-|B|\right). 
\end{align}
From equations (\ref{correct1}) and (\ref{correct2}) we obtain
\begin{align*}
\sigma(v)-\sigma(u) &=\big|\sigma(v|B)-\sigma(u|B)\big|+\big|\sigma(v|V(G)\setminus B)-\sigma(u|V(G)\setminus B)\big|\\
&\leq pn-(p-2)|B|.
\end{align*}
Since $pn-(p-2)|B|$ decreases as $|B|$ increases, we first determine a lower bound for $|B|$. Since the sets $B_{4i}$ are pairwise disjoint and $|B_{4i}|=4\delta$ for  $i=0,1,\ldots, \big\lfloor\frac{p-5}{8}\big\rfloor$ we have that 
\[|B|=\sum_{i=0}^{\big\lfloor\frac{p-5}{8}\big\rfloor} \big| B_{4i}\big|
=\left(\big\lfloor\frac{p-5}{8}\big\rfloor +1\right)4\delta
\geq \frac{(p-4)\delta}{2}.\]
This implies 
\[\sigma(v)-\sigma(u)\leq pn-(p-2)|B|\leq pn-\frac{(p-2)(p-4)\delta}{2}.\]
Finally, since  $pn-\frac{(p-2)(p-4)\delta}{2}$ is maximised for $p=\frac{n}{\delta}+3$, substituting this value of $p$ gives,
\[ \sigma(v)-\sigma(u) 
  \leq  \frac{n^2}{2\delta} + 3n + \frac{\delta}{2} 
  = (n-1)( \frac{n+1}{2\delta}+3)  + 3 + \frac{\delta}{2} + \frac{1}{2\delta}.  \]
Now $\frac{\delta}{2} + \frac{1}{2\delta} \leq \frac{n-1}{2}$ since $G$ is triangle-free, and 
$3 \leq \frac{n-1}{2}$ since $n\geq 7$.  
Dividing by $n-1$ thus yields , 
\[ 
\rho(G) - \pi(G)  =  \bar{\sigma}(v)-\bar{\sigma}(u)  
   \leq  \frac{n+1}{2\delta} + 4,  \]
as desired.\qedhere
}
\end{proof}

The bound in Theorem \ref{thm3.2} is sharp apart from an additive constant.
This can be seen by considering the graphs constructed in the following example.  

\begin{example}\label{ex4.2.1}
{\rm
Let $\delta \in \mathbb{N}$ be fixed with $\delta \geq 3$. For $k\in \mathbb{N}$ with $k$ even let $G_{\delta,k}$ be the sequential sum  
\[  \overline{K}_{1} + \overline{K}_{\delta} +  \overline{K}_{\delta -1} +\overline{K}_{1}+
\left[\overline{K}_{1}+\overline{K}_{\delta -1}+\overline{K}_{\delta -1}+\overline{K}_{1}\right]^{k-2} 
+ \overline{K}_{1} +\overline{K}_{\delta} +  \overline{K}_{\delta -1} +\overline{K}_{1}.   \]
Clearly, $G_{\delta,k}$ is triangle-free has order $n=2k\delta+2$, minimum degree $\delta$,  
diameter $4k-1 =\frac{2n-4}{\delta}-1$ and radius $2k= \frac{n-2}{\delta}$. 

Clearly the median vertices are the two centre vertices of $G_{\delta,k}$, and the margin vertices 
of $G_{\delta,k}$ are the two vertices in the first and last $\overline{K_1}$.
Tedious but straightforward calculations show that 
the distance of a median vertex equals $2\delta k^2 + 4k-3$, and the distance of a 
margin vertex equals $(2\delta k +2)(2k-\frac{1}{2})$, and so 
\[ \pi(G_{\delta,k}) =  \frac{n+1}{2\delta} - \frac{6\delta+3}{2\delta(n-1)} \quad \textrm{and} \quad \rho(G_{\delta,k}) = \frac{n-\delta/2-1}{\delta} - \frac{\delta+2}{2\delta(n-1)}. \]
Hence
\[\rho(G_{\delta, k})-\pi(G_{\delta, k})=\frac{n-\delta-3}{2\delta}+ \frac{5\delta+1}{2\delta(n-1)},\]
which differs from the bound in Theorem \ref{thm3.2}  by not more than $\frac{31}{6}$. \\
}
\end{example}

We now give a similar bound on the difference between remoteness and proximity in $C_4$-free graphs.
This bound shows that the bound in Theorem \ref{thm1.4} can be improved significantly for graphs not containing $4$-cycles as (not necessarily induced) subgraphs. 
In our proof we make use of the following well-known lemma. For a proof see, for example,  
\cite{er}.

\begin{lemma} {\rm (Erd\H{o}s, Pach, Pollack and Tuza \cite{er})} \label{la:order-of-N2-in-C4-free}
Let $G$ be a $C_4$-free graph of minimum degree $\delta$ and $v$ a vertex of $G$. Then 
$|N_{\leq 2}(v)| \geq \delta^2 - 2\lfloor \frac{\delta}{2} \rfloor + 1$. 
\end{lemma}

The proof of the following Theorem \ref{thm3.3} is similar to the proof of Theorem \ref{thm3.2}, 
hence we omit some of the details. 

\begin{theorem}\label{thm3.3}
Let $n, \delta\in\mathbb{N}$, with $\delta\geq 3$ and $n\geq 6$. If $G$ is a connected,  $C_4$-free 
graph of order $n$ and minimum degree $\delta$, then
\[\rho(G)-\pi(G)\leq \frac{5(n+1)}{4\left(\delta^2-2\big\lfloor\frac{\delta}{2}\big\rfloor +1\right)}
         +\frac{101}{20}.\]
\end{theorem} 
\begin{proof}
{\rm
Let $u$, $v$, and  $P: u_0,u_1,\ldots, u_{p-1}, u_p$ be as in the proof of Theorem \ref{thm3.2}. 
We define a subset $B$ of $V(G)$ for which $|\sigma(v|B)-\sigma(u|B)|$ is small. 
For $i\in\{0,1,\ldots,p\}$ let $A_i$ be a subset of $N_{\leq 2}(u_i)$ with exactly 
$\delta^2-2\big\lfloor\frac{\delta}{2}\big\rfloor +1$ vertices (which exists 
by Lemma \ref{la:order-of-N2-in-C4-free}) and we let $B_i=A_{i}\cup A_{p-i}$. 
Clearly, $A_i\cap A_j=\emptyset$ if $\big| j- i\big|\geq 5$ since otherwise $P$ is not a 
shortest $(v,u)$-path in $G$, and so $B_{5i}$, $i=0,1,\ldots, \big\lfloor\frac{p-5}{10}\big\rfloor$ 
are disjoint sets of cardinality $2\delta^2 - 4\lfloor \frac{\delta}{2} \rfloor +2$.  
Let $B$ be the set $\bigcup_{i=0}^{\lfloor\frac{p-5}{10}\rfloor} B_i$. 
Arguments similar to those in the proof of Theorem \ref{thm3.2} show that 
\[ 
\big| \sigma(v|B_{5i}) - \sigma(u,B_{5i}) \big|  \leq  4 |B_{5i}|.
\]
Summing this inequality over all $i \in \{0,1,\ldots, \big\lfloor\frac{p-5}{10}\big\rfloor\}$ 
yields that 
\begin{equation}\label{correct11}
\big|\sigma(v|B)-\sigma(u|B)\big|\leq 4|B|.
\end{equation} 
Similar to (2) in the proof of Theorem \ref{thm3.2} we obtain that
\[\big|\sigma(v|V(G)\setminus B)-\sigma(u|V(G)\setminus B)\big| \leq p(n-|B|). \tag{5} \label{correct22}\]
From equations (\ref{correct11}) and (\ref{correct22}) we obtain
\begin{align*}
\sigma(v)-\sigma(u) &=\big|\sigma(v|B)-\sigma(u|B)\big|+\big|\sigma(v|V(G)\setminus B)-\sigma(u|V(G)\setminus B)\big|\\
&\leq pn-(p-4)|B|.
\end{align*}
Since $pn-(p-4)|B|$ decreases as $|B|$ increases we first determine the smallest possible value for $|B|$. Since the sets $B_{5i}$ are pairwise disjoint and $|B_i|=2\delta^2-4\big\lfloor\frac{\delta}{2}\big\rfloor +2$ for $i=0,1,\ldots, \big\lfloor\frac{p-5}{10}\big\rfloor$, we have that 
\[ |B| =  \sum_{i=0}^{\big\lfloor\frac{p-5}{10}\big\rfloor} | B_{4i} |
   =
   \big(\big\lfloor\frac{p-5}{10}\big\rfloor +1\big)(2\delta^2-4\big\lfloor\frac{\delta}{2}\big\rfloor +2)
   \geq  \frac{(p-4)(\delta^2-2\big\lfloor\frac{\delta}{2}\big\rfloor +1)}{5}.    \]
This implies that 
\[\sigma(v)-\sigma(u)\leq pn-(p-4)|B|\leq pn-\frac{(p-4)^2(\delta^2-2\big\lfloor\frac{\delta}{2}\big\rfloor +1)}{5}.\]
Since  $pn-\frac{(p-4)^2(\delta^2-2\big\lfloor\frac{\delta}{2}\big\rfloor +1)}{5}$ is 
maximised for $p=\frac{5n}{2\left(\delta^2-2\big\lfloor\frac{\delta}{2}\big\rfloor +1\right)}+4$, substituting this value of $p$ 
and dividing by $n-1$ gives the desired bound as follows, 
\begin{align*}
\bar{\sigma}(v)-\bar{\sigma}(u)&\leq \frac{5n^2+16n\left(\delta^2-2\big\lfloor\frac{\delta}{2}\big\rfloor +1\right)}{4\left(\delta^2-2\big\lfloor\frac{\delta}{2}\big\rfloor +1\right)(n-1)}\\
&=\frac{5(n+1)}{4\left(\delta^2-2\big\lfloor\frac{\delta}{2}\big\rfloor +1\right)}+4+\frac{16\left(\delta^2-2\big\lfloor\frac{\delta}{2}\big\rfloor +1\right)+5}{4\left(\delta^2-2\big\lfloor\frac{\delta}{2}\big\rfloor +1\right)(n-1)}.
\end{align*}
Since $5 \leq \delta^2 - 2\lfloor \frac{\delta}{2} \rfloor +1$ we bound 
$\frac{16\left(\delta^2-2\big\lfloor\frac{\delta}{2}\big\rfloor +1\right)+5}{4\left(\delta^2-2\big\lfloor\frac{\delta}{2}\big\rfloor +1\right)(n-1)}\leq\frac{21}{4(n-1)}\leq \frac{21}{20}$ and obtain
\[\bar{\sigma}(v)-\bar{\sigma}(u)\leq \frac{5(n+1)}{4\left(\delta^2-2\big\lfloor\frac{\delta}{2}\big\rfloor +1\right)}+\frac{101}{20},\]
as desired.\qedhere
}
\end{proof}

The graphs constructed in the following example show that for $\delta+1$ a prime power, the bound 
in Theorem \ref{thm3.3} is close to being best possible in the sense that the ratio of the 
coefficients of $n$ in the bound and in the example below approach $1$ as $\delta$ gets large.

\begin{example}\label{ex4.2.2}
{\rm 
The construction of the following graph $H_{q,k}$ is due to Erd\H{o}s, Pach,
Pollack and Tuza \cite{er}. We summarise this construction here again for completeness.\\
Let $\delta \in \mathbb{N}$ be fixed with $\delta \geq 3$ such that $\delta = q-1$ for some prime power 
$q$. For $k\in \mathbb{N}$ with $k$ even define the graph $H_q$ as follows. The vertices of $H_q$ are 
the one-dimensional subspaces of the vector space $GF(q)^3$ over $GF(q)$, the finite field of order $q$. 
Two vertices are adjacent if, as subspaces, they are orthogonal. It is easy to verify that $H_q$ has 
$q^2+q+1$ vertices, that every vertex has either degree $q+1$ (if the corresponding subspace is not self-orthogonal) or $q$ (if the corresponding subspace is self-orthogonal), and that $H_q$ is $C_4$-free. 

Now choose a vertex $z$ of $H_q$ corresponding to a self-orthogonal subspace, and two neighbours $u$ and $v$ of $z$. It is easy to verify that $u$ and $v$ correspond to subspaces that are not self-orthogonal, and that $u$ and  $v$ are non-adjacent in $H_q$.  It is now easy to see that the set $M$ of edges joining a vertex in $N(u)-\{z\}$ to a vertex in $N(v)-\{z\}$ form a perfect matching between these two vertex sets. Since $z$ is the only common neighbour of $u$ and $v$ in $H_q$, and since removing $M$ destroys all $(u,v)$-paths of length three, the distance between $u$ and $v$ in $H_q-z-M$ is at least four. Let $H_q'$ be the graph $H_q-z-M$. Then $H_q'$ has order $q^2+q$, minimum degree $q-1$ and diameter at least $4$. It is not hard to show that its diameter equals $4$.

Let $k \in \mathbb{N}$ with $k\geq 2$. Let $G_1, G_2,\ldots,G_k$ be disjoint copies of the graph 
$H_q'$ and let $u_i$ and $v_i$ be the vertices of $G_i$ corresponding to $u$ and $v$, respectively, 
of $H_q'$. Define $H_{q, k}$ to be the graph obtained from $\bigcup_{i=1}^k G_i$ by adding the edges 
$v_iu_{i+1}$ for $i=1,2,\ldots,k-1$. Now clearly, since $H_q$ is $C_4$-free, it is easy to see that 
$H_{q, k}$ is a $C_4$-free graph of order $n=k(q^2+q) = k(\delta^2 + 3\delta +2)$ and minimum degree 
$\delta$. Then it is easy to verify that $H_{q,k}$ has diameter $5k-1$, radius $\frac{5}{2}k$, and that, 
for constant $\delta$ and large $k$,   
\[\pi(H_{q, k}) = \frac{5}{4}k +{\cal O}(1) = \frac{5}{4}\frac{n}{\delta^2 +3\delta +2}   + {\cal O}(1) \] and 
\[ \rho(H_{q, k}) = \frac{5}{2}k +{\cal O}(1) 
                     = \frac{5}{2}\frac{n}{\delta^2 +3\delta +2}   + {\cal O}(1).\] 
Hence
\[\rho(H_{q, k})-\pi(H_{q, k})=\frac{5}{4}\frac{n}{\delta^2 +3\delta +2}   + {\cal O}(1).\]
Since the bound in Theorem \ref{thm3.3} also equals 
$\frac{5(n+1)}{4\left(\delta^2-2\big\lfloor\frac{\delta}{2}\big\rfloor +1\right)}+{\cal O}(1)$, 
we have that the ratio of the coefficients of $n$ in the bound and in 
the above example equals 
$ \frac{ \delta^2-2\big\lfloor\frac{\delta}{2}\big\rfloor +1 }{ \delta^2 +3\delta +2}$, 
which approaches $1$ as $\delta$ gets large.\\
}
\end{example}

\section{Diameter vs proximity}

In this section we first give a lower bound on the proximity of a triangle-free graph of given order and 
diameter. As a corollary, we obtain an upper bound on the difference between diameter and proximity 
for triangle-free graphs, which strengthens the bound in Theorem \ref{thm1.5} by a factor of 
about $\frac{2}{3} \frac{\delta+1}{\delta}$. Both bounds are sharp apart from an additive constant.
We also prove corresponding bounds for $C_4$-free graphs.

To prove our upper bounds on the difference between diameter and proximity, we make use of the 
following results by  Erd\H{o}s et al. \cite{er}. 

\begin{theorem} {\rm (Erd\H{o}s, Pach, Pollack and Tuza \cite{er})} \label{Dtri} 
Let $G$ be a connected graph of order $n$ and minimum degree $\delta\geq 3$. \\
If $G$ is triangle-free, then 
\[{\rm diam}(G)\leq 4\Big\lceil\frac{n-\delta-1}{2\delta}\Big\rceil. \]
If $G$ is $C_4$-free, then 
\[{\rm diam}(G)\leq \Big\lfloor\frac{5n}{\delta^2-2\lfloor\frac{\delta}{2}\rfloor+1}\Big\rfloor.\]
\end{theorem}

\begin{theorem}\label{theo:smallest-proximity-given-diameter}
Let $n, \delta\in\mathbb{N}$, with $\delta\geq 3$ and $n\geq 8$. If $G$ is a connected,  
triangle-free graph of order $n$, minimum degree $\delta$ and diameter $d$, then
\[ \pi(G)  \geq \frac{ \delta (d-4)(d-1)}{8(n-1)}. \] 
\end{theorem} 

\begin{proof}
{\rm 
Let ${\rm diam}(G)=d$, and 
let $u$, $v_0$ and $v_d$ be vertices of $G$ with $\bar{\sigma}(u)=\pi(G)$ and 
$d_G(v_0,v_d)={\rm diam}(G)$.  Let $P: v_0,v_1,\ldots, v_{d-1}, v_d$ 
be a shortest $(v_0,v_d)$-path in $G$.  

For $i\in\{0,1,2,\ldots, d\}$ we define $A_i$ to be a subset of $N(v_i)$ that contains exactly 
$\delta$ vertices, and for $i\in \{1,2,\ldots,d-1\}$ we let 
$B_i=A_{i} \cup A_{i+1} \cup A_{d-1-i} \cup A_{d-i}$. 
Clearly, $A_{i} \cap A_{i+1} = \emptyset = A_{d-i-1} \cap A_{d-i}$, and 
$A_i\cap A_j=\emptyset$ if $|i-j|\geq 3$ since otherwise $P$ would not be a shortest 
$(v_0,v_d)$-path in $G$. We define $B_i = A_{i} \cup A_{i+1} \cup A_{d-1-i} \cup A_{d-i}$. 
Hence, the sets $B_{4i}$, $i=0,1, \ldots, \big\lfloor\frac{d-5}{8}\big\rfloor$, are disjoint.

We define $B$ to be the set $\bigcup_{i=0}^{\lfloor\frac{d-5}{8}\rfloor} B_{4i}$. 
To bound $\sigma(u|B)$ from below, we first consider $\sigma(u|A_j \cup A_{d-j})$ for $j=0,1,\ldots,d$. 
Let the elements of $A_j$ and $A_{d-j}$ be $w_1^{(j)}, w_2^{(j)},\ldots,w_{\delta}^{(j)}$ and 
$w_1^{(d-j)}, w_2^{(d-j)},\ldots,w_{\delta}^{(d-j)}$. For $t \in \{1,2,\ldots,\delta\}$ we have 
\[
d(u, w_t^{(j)}) + d(u, w_t^{(d-j)})  \geq   d(w_t^{(j)}, w_t^{(d-j)})
                  \geq d(v_j, v_{d-j}) - 2
                  = d-2j-2,
\]
with the last inequality holding since $u_j$ and $w_t^{(j)}$, and also $u_{d-j}$ and $w_t^{(d-j)}$,
are adjacent. Summing the last inequality over all $t\in \{1,2,\ldots,\delta\}$ yields that
\[   \sigma(u | A_j) + \sigma(u | A_{d-j}) \geq \delta(d-2j-2). \]
Hence, 
\begin{eqnarray*}
\sigma(u|B_{4i})&= &\sigma(u|A_{4i})+ \sigma(u,| A_{d-4i}) +\sigma(u|A_{4i+1}) + \sigma(u | A_{d-4i-1}) \\
&\geq  & \delta(d-8i-2) + \delta(d-8i-4) \\
& = &  2\delta d - 6 \delta - 16 \delta i.
\end{eqnarray*}
Summation over all $i\in\left\lbrace 0,1,\ldots, \Big\lfloor\frac{d-5}{8}\Big\rfloor\right\rbrace$ yields 
\begin{eqnarray*}
\sigma(u|B)& = & \sum_{i=0}^{\big\lfloor\frac{d-5}{8}\big\rfloor} \sigma(u|B_{4i}) \\
&\geq & \sum_{i=0}^{\big\lfloor\frac{d-5}{8}\big\rfloor}\left[2\delta d - 6 \delta - 16 \delta i\right]\\
&= & 2 \delta \left(\Big\lfloor\frac{d-5}{8}\Big\rfloor +1\right) 
    \left(d-3-4\Big\lfloor\frac{d-5}{8}\Big\rfloor\right).
\end{eqnarray*}    
Since $\lfloor\frac{d-5}{8}\rfloor \geq \frac{d-12}{8}$, this implies that
\[ \sigma(u|B)  \geq \frac{1}{8}\delta(d-4)(d-1). \]

Dividing by $n-1$ gives a lower bound on $\pi(G)$ as follows: 
\[ \pi(G) = \frac{1}{n-1} \sigma(u) \geq \frac{1}{n-1} \sigma(u | B) 
             \geq \frac{ \delta (d-4)(d-1)}{8(n-1)}. \] 
as desired.
}
\end{proof}

The following graph shows that for constant minimum degree $\delta$ and arbitrary $n,d$ 
the bound in Theorem \ref{theo:smallest-proximity-given-diameter} is sharp apart from an
additive constant.

\begin{example}\label{ex4.2.1-with-vertices-added}
{\rm
Let $\delta, d \in \mathbb{N}$ be fixed with $\delta \geq 3$ and $d \equiv 7 \pmod{8}$. 
Let $k=\frac{d+1}{4}$, so $k$ is even.  Let 
$n_0= 2k\delta +2$, i.e., $n_0$ is the order of the graph $G_{\delta,k}$ defined in 
Example \ref{ex4.2.1}. Fix a median vertex $u$ of $G_{\delta,k}$ and a neighbour $w$ of $u$
For $n \in \mathbb{N}$ with $n\geq n_0$ let $G^n_{\delta,k}$ be the graph obtained from
$G_{\delta,k}$ by adding $n-n_0$ vertices which are twins of $w$, i.e., each new vertex is 
adjacent to all neighbours of $w$ and to no other vertex. 

Clearly, the new graph has order $n$, minimum degree $\delta$, and no triangles. 
It is easy to verify that adding the new vertices does not change the diameter or radius, 
and that $u$ is a median vertex also of the new graph. From Example \ref{ex4.2.1} we thus get that
${\rm diam}(G^n_{\delta,k}) = 4k-1=d$, ${\rm rad}(G^n_{\delta,k}) = 2k$ and 
\[ \sigma_{G^n_{\delta,k}}(u) = \sigma_{G_{\delta,k}}(u) + n-n_0 
     = 2\delta k^2 + 4k-3 + n-(2k\delta+2). \]
Since $k=\frac{d+1}{4}$, we obtain by substituting this value and dividing by $n-1$ that
\[ \pi(G^n_{\delta,k}) = \frac{1}{n-1} \big( \frac{\delta}{8}(d-1)^2 + n+d - \frac{1}{2}\delta -4 \big). \]
Evaluating the difference between $\pi(G^n_{\delta,k})$ and the bound in 
Theorem \ref{theo:smallest-proximity-given-diameter} we obtain after simplifications that 
\[ \pi(G^n_{\delta,k}) - \frac{\delta}{8(n-1)}(d-4)(d-1) 
  = \frac{3 \delta d}{8(n-1)} + \frac{1}{n-1} \big(n+d - \frac{7}{8}\delta - 4\big).  \]
Now $\delta d = 2n_0 - (\delta+4) < 2(n-1)$, while 
$n+d -\frac{7}{8}\delta-4 < 2(n-1)$. Hence the difference between $\pi(G^n_{\delta,k})$ and the bound
in Theorem \ref{theo:smallest-proximity-given-diameter} cannot exceed $\frac{11}{4}$. 

Slight modifications of the construction of the graph $G^n_{\delta,k}$ yield graphs of any given 
diameter, not only for $d \equiv 7 \pmod{8}$, whose proximity differs from the bound in 
Theorem \ref{theo:smallest-proximity-given-diameter} by not more than a constant. 
This proves that the bound in Theorem \ref{theo:smallest-proximity-given-diameter} is sharp apart 
from an additive constant. \\ 
}
\end{example}

\begin{corollary}\label{thm4.3.2}
Let $n, \delta\in\mathbb{N}$, with $\delta\geq 3$ and $n\geq 8$. If $G$ is a connected,  triangle-free 
graph of order $n$ and minimum degree $\delta$, then
\[{\rm diam}(G)-\pi(G)\leq \frac{3(n-1)}{2\delta}+\frac{5}{2}.\]
This bound is sharp apart from an additive constant.
\end{corollary} 
             
\begin{proof}
{\rm   
Let $G$ be a connected,  triangle-free graph of order $n$ and minimum degree $\delta$. Denote the
diameter of $G$ by $d$. From Theorem \ref{theo:smallest-proximity-given-diameter} we have that 
\begin{equation} 
d-\pi(G)\leq d- \frac{ \delta (d-4)(d-1)}{8(n-1)}.
\label{eq:first-bound-on-d-minus-pi}
\end{equation}
Now $d \leq 4 \lceil \frac{n-\delta-1}{2\delta} \rceil$ by Theorem \ref{Dtri}, and so
$d < \frac{2(n-1)}{\delta}+2$. 
For such $d$, the derivative with respect to $d$ of the right hand side of 
\eqref{eq:first-bound-on-d-minus-pi} equals $1 - \frac{\delta(2d-5)}{8(n-1)}$, which is positive.
Hence the right hand side of \eqref{eq:first-bound-on-d-minus-pi} is increasing in $d$. 
Substituting $\frac{2(n-1)}{\delta}+ 2$ for $d$ yields, after simplification, that 
\[ {\rm diam}(G)-\pi(G)\leq  \frac{3(n-1)}{2\delta} + \frac{9}{4} + \frac{\delta}{4(n-1)}. \]
Since $\delta \leq n-1$, the corollary follows. \qedhere
}
\end{proof}

The graph in Example \ref{ex4.2.1} shows that the bound in Theorem \ref{thm4.3.2} is best possible, 
apart from an additive constant. \\

We now present a lower bound on the proximity in $C_4$-free graphs, given the order, minimum degree
and the diameter, and as a corollary we obtain an upper bound on the difference between the
diameter and the proximity of $C_4$-free graphs of given order and diameter.
The proofs are similar to the proofs of Theorem \ref{theo:smallest-proximity-given-diameter} 
and Corollary \ref{thm4.3.2}, hence we omit some of the details.

\begin{theorem}\label{theo:bound-on-proximity-given-diameter-in-C4free}
Let $n, \delta\in\mathbb{N}$, with $\delta\geq 3$ and $n\geq 8$. If $G$ is a connected,  $C_4$-free 
graph of order $n$ and minimum degree $\delta \geq 3$, then
\[  \pi(G)  \geq 
    \frac{\left(\delta^2-2\big\lfloor\frac{\delta}{2}\big\rfloor +1\right)(d-4)(d-3)}{20(n-1)}.  \]
\end{theorem}

\begin{proof}
{\rm
Let $d$, $u$, $v_0$, $v_d$ and $P:v_0, v_1,\ldots,v_d$ be as in the proof of 
Theorem \ref{theo:smallest-proximity-given-diameter}. 
 For $i\in\{1,2,\ldots,d\}$, we define $A_i$ to be a subset of $N_{\leq 2}(v_i)$ with exactly 
$\delta^2-2\big\lfloor\frac{\delta}{2}\big\rfloor +1$ vertices. Such a set exists by Lemma 
\ref{la:order-of-N2-in-C4-free}. For $i\in \{0,1,\ldots,d\}$ let 
$B_i=A_{i}\cup A_{d-i}$. Clearly, $A_i\cap A_j=\emptyset$ if $|j-i|\geq 5$ since otherwise $P$ is not a shortest $(v_0,v_d)$-path in $G$.  
Hence, the sets $B_{5i}$, $i=0,1,\ldots, \big\lfloor\frac{d-5}{10}\big\rfloor$ are disjoint.

We define $B$ to be the set $\bigcup_{i=0}^{\lfloor \frac{d-5}{10} \rfloor} B_i$. 
To bound $\sigma(u|B)$ from below, we first consider $\sigma(u|B_j)$ for $j\in \{0,1,\ldots,d\}$. 
Arguments similar to those in the proof of Theorem  \ref{thm4.3.2} show that 
\[ 
\sigma(u|B_i) \geq \left(\delta^2-2\big\lfloor\frac{\delta}{2}\big\rfloor +1\right)\left(d-2i-4\right).
\]
Summation over all $i\in \left\lbrace 0,1,\ldots, \Big\lfloor\frac{d-5}{10}\Big\rfloor\right\rbrace$ yields that 
\begin{eqnarray*}
\sigma(u|B)& = & \sum_{i=0}^{\lfloor (d-5) / 10 \rfloor} \sigma(u|B_{5i})\\
&\geq & \sum_{i=0}^{\lfloor (d-5) / 10 \rfloor}
  \left[ \left(\delta^2-2\big\lfloor\frac{\delta}{2}\big\rfloor +1\right) (d-10i-4)\right] \\
& = &  \left(\delta^2-2\big\lfloor\frac{\delta}{2}\big\rfloor +1\right) 
     \left( \big\lfloor \frac{d-5}{10} \big\rfloor + 1 \right) 
     \left( d - 4 - 5\big\lfloor \frac{d-5}{10} \big\rfloor \right).   
\end{eqnarray*}
Now $\lfloor \frac{d-5}{10} \rfloor + 1 \geq \frac{d-4}{10}$, and 
$5 \lfloor \frac{d-5}{10} \rfloor \leq \frac{d-5}{2}$. Hence 
\[ \sigma(u |B) \geq \left(\delta^2-2\big\lfloor\frac{\delta}{2}\big\rfloor +1\right) 
      \frac{(d-4)(d-3)}{20}.  \]
Dividing by $n-1$ yields 
\[ 
\pi(G) \geq \frac{1}{n-1} \sigma(u|B) 
 \geq \frac{\left(\delta^2-2\big\lfloor\frac{\delta}{2}\big\rfloor +1\right)(d-4)(d-3)}{20(n-1)},  \]
as desired. \qedhere
}
\end{proof}

\begin{corollary}\label{thm4.3.3}
Let $n, \delta\in\mathbb{N}$, with $\delta\geq 3$ and $n\geq 6$. If $G$ is a connected,  $C_4$-free 
graph of order $n$ and minimum degree $\delta$, then
\[{\rm diam}(G)-\pi(G)\leq 
  \frac{15n}{4\left(\delta^2-2\big\lfloor\frac{\delta}{2}\big\rfloor +1\right)}+\frac{7}{4}.\]
\end{corollary}

\begin{proof}
{\rm 
Let  $G$ be a connected,  $C_4$-free graph of order $n$ and minimum degree $\delta$. Denote the
diameter of $G$ by $d$. From Theorem \ref{theo:bound-on-proximity-given-diameter-in-C4free} 
we have that 
\[ 
d - \pi(G) \leq d - 
  \frac{\left(\delta^2-2\big\lfloor\frac{\delta}{2}\big\rfloor +1\right)(d-4)(d-3)}{20n}, \]
where we bounded $n-1$ in the denominator by $n$ for easier calculations. 
Now $d \leq \frac{5n}{\delta^2 - 2 \lfloor d/2 \rfloor +1}$ by Theorem \ref{Dtri}.

It is easy to verify that the right hand side of the above inequality is increasing in $d$ for 
$d \leq \frac{5n}{\delta^2 - 2 \lfloor d/2 \rfloor +1}$. Substituting this value yields,
after simplification, 
\begin{eqnarray*}
{\rm diam}(G)  - \pi(G) & \leq & \frac{15n}{4 \left(\delta^2-2\big\lfloor\frac{\delta}{2}\big\rfloor +1\right)} 
   + \frac{7}{4} - \frac{3\left(\delta^2-2\big\lfloor\frac{\delta}{2}\big\rfloor +1\right)}{5n}  \\
 & < & \frac{15n}{4 \left(\delta^2-2\big\lfloor\frac{\delta}{2}\big\rfloor +1\right)} 
   + \frac{7}{4}, 
\end{eqnarray*}   
as desired.\qedhere
}
\end{proof}

The graph $H_{q,k}$ in Example \ref{ex4.2.2} shows that the bound in Theorem \ref{thm4.3.3} is 
close to being best possible (in the same sense as discussed there).\\

\section{Radius and proximity}
In this section we determine, up to an additive constant, the minimum proximity of a triangle-free 
graph of given order, radius and minimum degree. As a corollary we obtain an upper bound on 
the difference between the radius and the proximity in triangle-free graphs of given order and
minimum degree which is sharp apart from an additive constant. 
We obtain similar results for $C_4$-free graphs. These results are close to best possible, in a sense
specified later. 
The proof strategy is similar to that of Theorem \ref{thm1.6} (see \cite{dan}).

\begin{definition} We define the distance between two edges $e$ and $f$ of a connected graph to 
be the smallest distance between an end of $e$ and an end of $f$. If $e=uv$, then $N(e)$ 
stands for $N(u) \cup N(v)$. 
\end{definition}

\begin{theorem}\label{lem4.4.2}
Let $n, \delta, r\in\mathbb{N}$ with $\delta\geq 3$, $r\geq 1$ and $n\geq 6$. 
If $G$ is a connected,  triangle-free graph of order $n$, minimum degree $\delta$ and radius $r$, then
\[\pi(G)\geq \frac{\delta}{2(n-1)}\left[r^2 - 7r +\frac{47}{8}\right]. \]
For constant $\delta$ this bound is sharp apart from an additive constant. 
\end{theorem} 

\begin{proof}
{\rm
Let $u$ be a vertex of $G$ with $\bar{\sigma}(u)=\pi(G)$. Let ${\rm ecc}(u)=R$ and ${\rm rad}(G)=r$. 
For $i=0,1,\ldots, R$, let $N_i$ be the set of vertices at distance $i$ from $u$. For each 
$i\in\{0,1,\ldots, R\}$ consider the set of edges joining a vertex in $N_i$ to a vertex in $N_{i+1}$,  
and let $A_i$ be a subset of maximum cardinality such that the distance in $G$ between any two edges 
in $A_i$ is at least $3$. Since $G$ is a triangle-free graph, we have $| N(e) | \geq 2\delta$ for all
$e \in A_i$.  This implies that for each $i \in \{1,2,\ldots,R-1\}$ we have that 
$\big| \bigcup_{e\in A_i} N(e) \big|\geq 2\delta|A_i|$. Since every vertex in 
$\bigcup_{e\in A_i} N(e) $ is at least distance $i-1$ from $u$ in $G$, we have 
\[\sum_{v\in N[A_i]} d(u, v)\geq \left(i-1\right)2\delta |A_i|.\] 
If $v$ is a vertex of $G$, then $v$ belongs to at most four of the sets $N[A_i]$,  
$i=1,2,\ldots, R$. Hence, we have that
\begin{equation}\label{4sigma}
4\sigma (u)\geq \sum_{i=1}^{R-1} \left(i-1\right)2\delta |A_i|.
\end{equation}
We consider two cases, depending on whether $R$ is much larger than $r$ or not.\\[1mm]
{\sc Case 1:} $R\geq \frac{1}{2}(3r-10)$.\\
Since $|A_i| \geq 1$ for all $i=1,2,\ldots, R-1$, we obtain from \eqref{4sigma} that
\[ 4\sigma (u) \geq  \sum_{i=1}^{R-1} \left(i-1\right) 2\delta |A_i| 
  \geq  \sum_{i=1}^{R-1} \left(i-1\right) 2\delta 
 =  \delta\left(R-2\right)\left(R-1\right). \]
By the defining condition of Case 1 we have $R \geq \frac{1}{2}(3r-10)$. Therefore, 
\[ 4 \sigma(u) \geq \delta \big(\frac{9}{4} r^2 - \frac{39}{2}r + 42\big). \]
Since $\pi(G) = \frac{1}{n-1} \sigma(u) $, we thus obtain 
\[ \pi(G) 
      \geq \frac{\delta}{4(n-1)} ( \frac{9}{4}r^2 - \frac{39}{2}r + 42) 
      \geq \frac{\delta}{2(n-1)} (r^2 - 7r + \frac{47}{8}), \]
where the last inequality is easy to verify with elementary calculations. 
Hence the theorem follows in this case.        
   \\[1mm]
{\sc Case 2:} $R\leq \frac{1}{2}(3r-11)$.\\
Clearly, $|A_i|\geq 1$ for $i=1,2,\ldots, R-1$. We claim that 
\begin{equation} \label{eq:Ai>=2}
\textrm{if $R-r+ 5 \leq i \leq 2r-R-6$, then  $|A_i| \geq 2$.}
\end{equation}
Note that the defining condition $R \leq \frac{1}{2}(3r-11)$ of this case guarantees that there 
exist $i$ with $R-r+ 5 \leq i \leq 2r-R-6$. 
To prove \eqref{eq:Ai>=2} we suppose to the contrary that there exists an integer $j$ with 
$R-r+5\leq j\leq 2r-R-6$ such that 
$|A_j|=1$. Now fix a vertex $u_R\in N_R$ and let $P:u, u_1, u_2, \ldots, u_R$ be a shortest 
$(u,u_R)$-path in $G$. Then $u_{R-r+5}$ is in $N_{R-r+5}$. We now show that 
$d_G(u_{R-r+5}, x)\leq r-1$ for any $x\in V(G)$. This contradiction to $r$ being the radius of $G$  
will prove \eqref{eq:Ai>=2}.\\
First consider the case $d_G(u, x)\geq j+1$.
Let $P'$ be a shortest $(u,x)$-path and let $v_j v_{j+1}$ be an edge in $P'$ with $v_j$ in $N_j$ 
and $v_{j+1}\in N_{i+1}$. Then $d(u_j, v_j) \leq 4$ since otherwise, if $d(u_j, v_j) \geq 5$, the 
edges $u_ju_{j+1}$ and $v_jv_{j+1}$ would be at distance at least $3$, contradicting the 
maximality of the set $A_j$ since $|A_j|=1$. Hence 
\begin{align*}
d_G(u_{R-r+5}, x)&\leq d_G(u_{R-r+5}, u_j)+d_G(u_j, v_j)+ d_G(v_j, x)\\
&\leq (j-R+r-5)+4+(R-j)\\
&= r-1.
\end{align*}
Now consider the case $d_G(u_0, x)\leq j$. Then
\begin{eqnarray*}
d_G(u_{R-r+5}, x)&\leq & d_G(u_{R-r+5}, u_0)+ d_G(u_0, x)\\
&\leq & (R-r+5)+j\\
&\leq & (R-r+5)+(2r-R-6)\\
&= & r-1.
\end{eqnarray*}
Hence  ${\rm ecc}(u_{R-r+5})<r = {\rm rad}(G)$. This contradiction completes the proof of 
\eqref{eq:Ai>=2}. \\
Inequality (\ref{4sigma}) in conjunction with \eqref{eq:Ai>=2} implies that 
\begin{eqnarray*}
4\sigma (u)&\geq & \sum_{i=1}^{R-1} \left(i-1\right)2\delta |A_i| \\
&\geq & \sum_{i=1}^{R-1} \left(i-1\right)2\delta + \sum_{i=R-r+5}^{2r-R-6} \left(i-1\right)2\delta\\ 
&= & \delta\left(R^2 + 3r^2 - 2rR -R -13r +12\right). 
\end{eqnarray*} 
Elementary calculus shows that for fixed $r$ the right hand side of the last inequality is minimised, 
as a function of $R$, if $R=r + \frac{1}{2}$. Since $R$ is an integer, 
the right hand side is maximised for $R=r$ as well as for $R=r+1$. Substituting either 
value for $R$ yields that 
\[ 4\pi(u) \geq \delta (2r^2-14r+12). \]
Since $\pi(G) = \frac{1}{n-1} \sigma(u) $, we thus obtain
\[ \pi(G) 
      \geq \frac{\delta}{4(n-1)} ( 2r^2 - 14r + 12) 
      > \frac{\delta}{2(n-1)} (r^2 - 7r + \frac{47}{8}), \] 
which is the desired bound.  \\
To see that the bound is sharp consider the graph $G^n_{\delta,k}$ in 
Example \ref{ex4.2.1-with-vertices-added}. It follows from our observations there that 
$\pi(G^n_{\delta,k}) = \frac{1}{n-1} \big( \frac{1}{2}\delta r^2 - \delta r + n-5 + 2r \big)$,
where $r$ denotes the radius of $G^n_{\delta,k}$. Since the radius is not more than the diameter,
we have by Theorem \ref{Dtri} that 
$\delta r  <2(n-1)+2\delta$. Using this inequality and the fact that $r \leq n/2$, it is 
easy to show that the difference between $\pi(G^n_{\delta,k})$ and our bound is not more
than $10$. Hence the bound is sharp apart from an additive constant.  \qedhere
}
\end{proof}

\begin{corollary}\label{thm4.4.2}
Let $n, \delta, r\in\mathbb{N}$, with $\delta\geq 3$ and $n\geq 6$. If $G$ is a connected,  
triangle-free graph of order $n$ and  minimum degree $\delta$, then
\[{\rm rad}(G)-\pi(G)\leq \frac{n-1}{2\delta}+\frac{11}{2},\]
and this bound is sharp apart from an additive constant. 
\end{corollary} 

\begin{proof}
{\rm
Denote the radius of $G$ by $r$. 
From Theorem \ref{lem4.4.2} we have that
\[
r-\pi(G)\leq r-\frac{\delta}{2(n-1)} \left( r^2 -7r + \frac{47}{8}\right).
\]
Elementary calculus shows that for constant $n$ and $\delta$ the right hand side of 
the above inequality is maximised for $r=\frac{n-1}{\delta}+\frac{7}{2}$. Substituting this value yields
that
\[ r - \pi(G) \leq \frac{n-1}{2\delta} + \frac{7}{2} + \frac{51\delta}{16(n-1)}. \]
Since it follows from $\delta \leq n/2$ and $n\geq 6$ that $\frac{51\delta}{16(n-1)} < 2$, 
the desired bound follows. \\
To see that the bound is sharp apart from an additive constant consider the graph 
$G_{\delta,k}$ in Example \ref{ex4.2.1}, where we established that 
$\pi(G_{\delta,k}) = \frac{n+1}{2\delta}  - \frac{6\delta+3}{2\delta(n-1)}$ and 
${\rm rad}(G_{\delta,k}) = \frac{n-2}{\delta}$, and so 
\[  {\rm rad}(G_{\delta,k}) - \pi(G_{\delta,k}) = \frac{n-2}{\delta} 
             - \frac{n+1}{2\delta}  + \frac{6\delta+3}{2\delta(n-1)} 
             =  \frac{n-1}{2\delta} + \frac{6\delta+5}{2\delta(n-1)},   \]
which differs from our bound by less than $4$.\qedhere
}
\end{proof}

We now present similar bounds for $C_4$-free graphs.

\begin{theorem}\label{theo:lower-bound-on-pi-given-radius-C4free}
Let $n, \delta\in\mathbb{N}$, with $\delta\geq 3$ and $n\geq 16$. If $G$ is a connected, $C_4$-free 
graph of order $n$, minimum degree $\delta$ and radius $r$, then
\[ \pi(G)\geq 
  \frac{\delta^2-2\lfloor\frac{\delta}{2}\rfloor+1}{5(n-1)} \big(r^2 - 8r + \frac{127}{8}\big). \]
\end{theorem} 

\begin{proof} 
{\rm
Let $u$, $R$, $r$ and $N_i$ be as in the proof of Theorem \ref{lem4.4.2}.  
For each $i\in\{0,1,\ldots, R\}$, let $A_i$ be a subset of $N_i$ of maximum cardinality such that 
the distance between any two vertices of $A_i$ is at least $5$ in $G$. Since $G$ is $C_4$-free, 
we have by Lemma \ref{la:order-of-N2-in-C4-free} that 
$|N_{\leq 2}(v)|\geq \delta^2-2\big\lfloor\frac{\delta}{2}\big\rfloor +1$, for all 
$v\in V(G)$. This implies that for each $A_i$ we have that 
$|N_{\leq 2}[A_i]|\geq \left(\delta^2-2\big\lfloor\frac{\delta}{2}\big\rfloor +1\right)|A_i|$. 
Also, since every vertex $v\in N_{\leq 2}[A_i]$ is at least distance $i-2$ from $u$ in $G$, 
\[\sum_{v\in N_{\leq 2}[A_i]} d(u, v)\geq \left(i-2\right)\left(\delta^2-2\big\lfloor\frac{\delta}{2}\big\rfloor +1\right)|A_i|.\] 
If $v$ is a vertex of $G$, then $v$ belongs to at most five of the sets $N_{\leq 2}[A_i]$, 
$i=2,3,\ldots, R$. Hence, we have that
\begin{equation}\label{5sigma}
5\sigma(u) \geq \sum_{i=2}^{R} \left(i-2\right)\left(\delta^2-2\big\lfloor\frac{\delta}{2}\big\rfloor +1\right)|A_i|.
\end{equation}
{\sc Case 1:} $R\geq \frac{1}{2}(3r-9)$.\\
Since each $|A_i| \geq 1$ for all $i=2,3,\ldots, R$, it follows from (\ref{5sigma}) that 
\begin{align*}
5\sigma (u)&\geq 
 \sum_{i=2}^{R} \left(i-2\right)\left(\delta^2-2\big\lfloor\frac{\delta}{2}\big\rfloor +1\right)|A_i|\\
&\geq 
 \sum_{i=2}^{R} \left(i-2\right)\left(\delta^2-2\big\lfloor\frac{\delta}{2}\big\rfloor +1\right)\\
&= \frac{1}{2}\left(\delta^2-2\big\lfloor\frac{\delta}{2}\big\rfloor +1\right)\left(R^2-3R +2 \right)\\
&\geq 
\left(\delta^2-2\big\lfloor\frac{\delta}{2}\big\rfloor +1\right)\left(\frac{9}{8}r^2-9r+\frac{143}{8} \right)\\
&> \left(\delta^2-2\big\lfloor\frac{\delta}{2}\big\rfloor +1\right)\left(r^2-8r + \frac{127}{8}\right). 
\end{align*}
Since $\pi(G)=\frac{1}{n-1} \sigma(u)$, dividing by $5(n-1)$ yields the desired bound on $\pi(G)$. \\[1mm]
{\sc Case 2:} $R\leq \frac{1}{2}(3r-10)$.\\
Clearly, $|A_i|\geq 1$ for $i=2,3,\ldots, R$. Arguments similar to the proof of \eqref{eq:Ai>=2} 
in Theorem \ref{lem4.4.2} show that  
\begin{equation} \label{eq:Ai>=2-in-C4free}
\textrm{if $R-r+ 5 \leq i \leq 2r-R-5$, then  $|A_i| \geq 2$.}
\end{equation}
(Note that the defining condition $R \leq \frac{1}{2}(3r-10)$ of this case guarantees that there 
exist $i$ with $R-r+ 5 \leq i \leq 2r-R-5$.)  
We obtain from (\ref{5sigma}) that
\begin{align*}
5\sigma (u)&\geq \sum_{i=2}^{R} \left(i-2\right)\left(\delta^2-2\big\lfloor\frac{\delta}{2}\big\rfloor +1\right)|A_i|\\
&\geq \sum_{i=2}^{R} \left(i-2\right)\left(\delta^2-2\big\lfloor\frac{\delta}{2}\big\rfloor +1\right) + \sum_{i=R-r+5}^{2r-R-5} \left(i-2\right)\left(\delta^2-2\big\lfloor\frac{\delta}{2}\big\rfloor +1\right)\\
&= \frac{1}{2}\left(\delta^2-2\big\lfloor\frac{\delta}{2}\big\rfloor +1\right)\left(2r^2+(R-r)^2+5R-21r+38\right)\\
&\geq \left(\delta^2-2\big\lfloor\frac{\delta}{2}\big\rfloor +1\right)\left(r^2-8r+19\right)\\
&> \left(\delta^2-2\big\lfloor\frac{\delta}{2}\big\rfloor +1\right)\left[r^2 - 8r +\frac{127}{8}\right].
\end{align*}
Since $\pi(G)=\frac{1}{n-1} \sigma(u)$, dividing by $5(n-1)$ yields the desired bound on $\pi(G)$ 
also in this case.   \qedhere
}
\end{proof}

\begin{corollary}\label{coro:radius-prox-for-C4-free}
Let $n, \delta\in\mathbb{N}$, with $\delta\geq 3$ and $n\geq 16$. If $G$ is a connected,  $C_4$-free 
graph of order $n$ and minimum degree $\delta$, then
\[\mbox{rad}(G)-\pi(G)\leq \frac{5(n-1)}{4\left(\delta^2-2\lfloor\frac{\delta}{2}\rfloor+1\right)}+4.\]
\end{corollary} 

We omit the proof of Corollary \ref{coro:radius-prox-for-C4-free} as it is almost identical to the 
proof of Corollary \ref{thm4.4.2}. 
The graph in Example \ref{ex4.2.2} shows that the bound in Corollary \ref{coro:radius-prox-for-C4-free} 
is close to being best possible in the sense described there.

\end{document}